\documentclass[a4paper, 11pt]{amsart}

 \usepackage{mathrsfs,amsfonts,amsmath,amssymb,amsthm}
 \usepackage{color}

 \makeatletter
 \@addtoreset{equation}{section}
 \makeatother

\newtheorem{theorem}{Theorem}[section]
\newtheorem{lemma}[theorem]{Lemma}       
\newtheorem{corollary}[theorem]{Corollary}
\newtheorem{proposition}[theorem]{Proposition}

\theoremstyle{remark}
\newtheorem{remark}[theorem]{Remark}
\newtheorem{example}[theorem]{Example}

 \def\beqlb{\begin{eqnarray}}\def\eeqlb{\end{eqnarray}}
 \def\beqnn{\begin{eqnarray*}}\def\eeqnn{\end{eqnarray*}}

 \def\mbb{\mathbb}

 \def\qed{\hfill$\Box$\medskip}

\newcommand{\bcen}{\begin{center}}
\newcommand{\ecen}{\end{center}}
\newcommand{\bgeqn}{\begin{equation}}
\newcommand{\edeqn}{\end{equation}}

\newcommand{\cal}{\mathcal}

\def\dz{\delta}

\def\rar{\rightarrow}

\renewcommand{\P}{\mathbb{P}}
\newcommand{\N}{\mathbb{N}}
\newcommand{\E}{\mathbb{E}}

\begin{document}

\title[Tree-valued
Markov Processes]{Pruning Galton-Watson Trees and Tree-valued
Markov Processes}

\author{Romain Abraham} 

\address{
Romain Abraham,
MAPMO, CNRS UMR 6628,
F\'ed\'eration Denis Poisson FR 2964,
Universit\'e d'Orl\'eans,
B.P. 6759,
45067 Orl\'eans cedex 2
FRANCE.
}
  
\email{romain.abraham@univ-orleans.fr} 

\author{Jean-Fran\c cois Delmas}

\address{
Jean-Fran\c cois Delmas,
CERMICS,  Universit\'e Paris-Est, 6-8
av. Blaise Pascal, 
  Champs-sur-Marne, 77455 Marne La Vall\'ee, FRANCE.}

\email{delmas@cermics.enpc.fr}

\author{Hui He} 

\address{
Hui He,
MAPMO, CNRS UMR 6628,
F\'ed\'eration Denis Poisson FR 2964,
Universit\'e d'Orl\'eans,
B.P. 6759,
45067 Orl\'eans cedex 2
FRANCE
}

\address{and}

\address{
School of Mathematical Sciences,
Beijing Normal University,
Beijing 100875, P.R.China.
}
  
\email{hehui@bnu.edu.cn} 

\thanks{This work is partially supported by the ``Agence Nationale de
  la Recherche'', ANR-08-BLAN-0190, and by NSFC (10525103 and 10721091).}

\begin{abstract}
We present a new pruning procedure on discrete trees by adding marks
on the nodes of trees. This procedure allows us to construct and
study a tree-valued Markov process $\{ {\cal G}(u)\}$ by pruning
Galton-Watson trees and an analogous process $\{{\cal G}^*(u)\}$ by
pruning a critical or subcritical Galton-Watson tree conditioned to
be infinite. Under a mild condition on offspring distributions, we
show that the process $\{{\cal G}(u)\}$ run until its ascension time
has a representation in terms of $\{{\cal G}^*(u)\}$. A similar
result was obtained by Aldous and Pitman (1998) in the special case
of Poisson offspring distributions where they considered uniform
pruning of Galton-Watson trees by adding marks on the edges of
trees.
\end{abstract}



\maketitle

\section{Introduction}

Using percolation on the branches of a Galton-Watson tree, Aldous and
Pitman constructed by time-reversal in \cite{[AP98]} an inhomogeneous tree-valued Markov process that
starts from the trivial tree consisting only of the root and ends at
time 1 at the initial Galton-Watson tree. When the final Galton-Watson
tree is infinite, they define the ascension time $A$ as the first
time where the tree becomes infinite. They also define another process
by pruning at branches the tree conditioned on non-extinction and they show that,
in the special case of Poisson offspring distribution, some connections
exist between the first process up to the ascension time and the
second process.

Using the same kind of ideas, continuum-tree-valued Markov processes are
constructed in \cite{[AD10]} and an analogous relation is exhibited
between the process obtained by pruning the tree and the other one
obtained by pruning the tree conditioned on non-extinction. However,
in that continuous framework, such results hold under very general
assumptions on the
branching mechanism.

Using the ideas of the pruning procedure \cite{[AD10]} (which first
appeared in \cite{[M05]} for a different purpose), we propose here
to prune a Galton-Watson tree on the nodes instead of the branches so
that the connections pointed out in \cite{[AP98]} hold for any offspring distribution.

Let us first explain the pruning procedure.
Given a probability distribution $p=\{p_n, n=0,1,\cdots\}$, let ${\cal
G}_p$ be a Galton Watson tree with offspring distribution $p$. Let
$0<u<1$ be a constant. Then, if $\nu$ is an inner node of ${\cal
G}_p$ that has $n$ offsprings, we cut it (and discard all the sub-trees
attached at this node) with probability $u^{n-1}$
independently of the other nodes. The resulting tree  is
still a Galton-Watson tree  with offspring distribution $p^{(u)}$
defined by:
\begin{equation}\label{eq:def_pu_1}
p_n^{(u)}=u^{n-1} p_n,\quad \text{for } n\geq1
\end{equation}
and
\begin{equation}\label{eq:def_pu_2}
p_0^{(u)}=1-\sum_{n=1}^{\infty}p_n^{(u)}.
\end{equation}
This particular pruning is motivated by the following lemma whose
proof is postponed to Section \ref{sec:appendix}.

\begin{lemma}\label{lem:leaves}
Let $p$ and $q$ be two offspring distributions. Let ${\cal
G}_p$ and ${\cal G}_q$ be the associated Galton-Watson trees and let
$\#{\cal L}_{p}$ and $\#{\cal L}_q$ denote the number of leaves of
${\cal G}_p$ and ${\cal G}_q$, respectively. Then we have that
\begin{equation}\label{eqnGW}
\forall N\ge 1,\quad \P({\cal G}_p\in \cdot|\#{\cal L}_p=N)=\P({\cal G}_q\in \cdot|\#{\cal
L}_q=N)
\end{equation}
if and only if
$$\exists u>0,\ \forall n\ge 1,\ q_n=u^{n-1}p_n.$$

\end{lemma}

This lemma can be viewed as the discrete analogue of Lemma 1.6 of
\cite{[AD08]} that explains the choice of the pruning parameters for
the continuous case. In \cite{[Al91]}, a similar result for Poisson
Galton-Watson trees was obtained when conditioning by the total
number of vertices, which explains why Poisson-Galton-Watson trees play
a key role in \cite{[AP98]}.

Using the pruning at nodes procedure, given a critical offspring distribution
$p$, we construct in Section 4 a tree-valued (inhomogeneous)
Markov processes
$({\cal G}(u),0\le u\le \bar u)$ for some $\bar u\ge 1$, such that
\begin{itemize}
\item the process is non-decreasing,
\item for every $u$, ${\cal G}(u)$ is a Galton-Watson tree with
  offspring distribution $p^{(u)}$,
\item the tree is critical for $u=1$, sub-critical for $u<1$ and
  super-critical for $1<u\le\bar u$.
\end{itemize}

Let us state the main properties that we prove for that process and
let us compare them with the results of \cite{[AP98]}. We write
$(\mathcal{G}^{AP}(u))$ for the tree-valued Markov process defined in \cite{[AP98]}.

In Section \ref{sec:processes}, we compute the forward transition
probabilities and the forward transition rates for that process and
exhibit a martingale that will appear several times (see Corollary
\ref{CorL}). For a tree $\mathbf{t}$, we set
\begin{equation}\label{eq:defh}
M(u,\mathbf{t})=\frac{(1-\mu(u))\#{\cal L}(\mathbf{t})}{p_0^{(u)}}
\end{equation}
where $\#{\cal L}(\mathbf{t})$ denotes the number of leaves of
  $\mathbf{t}$ and 
$\mu(u)$ if the mean of offsprings in $\mathcal{G}(u)$. Then, the
  process
$$\bigl(M(u,\mathcal{G}(u)),0\le u<1\bigr)$$
is a martingale with respect to the filtration generated by $\mathcal{G}$.
In \cite{[AP98]}, the martingale that appears (Corollary 23) for
Poisson-Galton-Watson trees is $(1-\mu(u))\#\mathcal{G}^{AP}(u)$.

When the tree $\mathcal{G}(u)$ is super-critical, it may be infinite. We
define the ascension time $A$ by:
$$A=\inf\{u\in[0,\bar u],\#\mathcal{G}(u)=\infty\}$$
with the convention $\inf\emptyset=\bar u$.
We can then compute the joint law of $A$
and ${\cal G}_{A-}$ (i.e. the tree just before it becomes infinite),
see Proposition \ref{PropA}: we set $F(u)$ the extinction
probability of $\mathcal{G}(u)$ and we have for $u\in[0,\bar u)$
\begin{align*}
\P(A\le u) & =1- F(u)\\
\P({\cal G}(A-)=\mathbf{t}\bigm|A=u) & = M(\hat u,\mathbf{t})\P({\cal
  G}(\hat u)=\mathbf{t})
\end{align*}
with $\hat u=u F(u)$. These results are quite similar with those of
Lemma 22 of \cite{[AP98]}. They must also be compared to the
continuous framework, Theorem 6.5, Theorem
6.7 of \cite{[AD10]}.

When we have $p_0^{(\bar u)}=0$, then the final tree $\mathcal{G}(\bar
u)$ is a.s. infinite. We consider the tree ${\cal G}^*(1)$ which is distributed as the tree
$\mathcal{G}(1)$  conditioned on non-extinction. From
this tree, by the same pruning procedure, we construct a
non-decreasing tree-valued process $({\cal G}^*(u),0\le u\le 1)$. We
then prove the following representation formula (Proposition \ref{PropRepA}):
$$({\cal G}(u),\  0\leq
u<A)\overset{d}{=} ({\cal G}^*(u \gamma),\ 0\leq
u<\bar{F}^{-1}(1-\gamma)),$$
where $\gamma$ is a r.v, uniformly distributed on $(0,1)$,
independent of $\{{\cal G}^{*}(\alpha):0\leq \alpha\leq1\}$ and $\bar F(u)=1-F(u)$. This
result must also be compared to a similar result in \cite{[AP98]},
Proposition 26:
$$({\cal G}^{AP}(u),\ 0\leq
u<A)\overset{d}{=} \left({\cal G}^{AP*}(u \gamma),\ 0\leq
u<\frac{-\log\gamma}{(1-\gamma)}\right),$$
or to its continuous analogue, Corollary 8.2 of \cite{[AD10]}.

Let us stress again that, although the results are very similar, those
in \cite{[AP98]} only hold for Poisson-Galton-Watson trees whereas the
results presented here hold for any offspring distribution.

The paper is organized as follows. In the next section, we recall some
notation for trees and define the pruning procedure at nodes. In
Section \ref{sec:processes}, we define the processes ${\cal G}$ and
${\cal G^*}$ and in Section \ref{sec:representation} we state and
prove the main results of the paper. Finally, we prove Lemma
\ref{lem:leaves} in
Section \ref{sec:appendix} .

\section{Trees and Pruning}

\subsection{Notation for Trees}

We present the framework developed in \cite{[Ne86]} for trees, see
also \cite{[LeG05]} or \cite{[AP98]} for more notation and terminology. Introduce the set of labels
$${\cal W}=\bigcup_{n=0}^{\infty}(\mathbb{N}^*)^n,
$$
where $\mbb{N}^*=\{1,2,\ldots\}$ and by convention
$(\mbb{N}^*)^0=\{\emptyset\}.$

An element of $\cal{W}$ is thus a sequence $w = (w^1, \ldots , w^n)$
of elements of $\mbb{N}$, and we set $|w|=n$, so that $|w|$
represents the generation of $w$ or the height of $w$. If $w = (w^1,
\ldots, w^m)$ and $v = (v^1,\ldots, v^n)$ belong to $\cal W$, we
write $wv = (w^1,\ldots,w^m, v^1,\ldots, v^n)$ for the concatenation
of $w$ and $v$. In particular $w\emptyset =\emptyset w = w$. The
mapping $\pi:{\cal W}\setminus\{\emptyset\}\longrightarrow{\cal W}$
is defined by $\pi((w^1,\ldots,w^n)) = (w^1,\ldots,w^{n-1})$ if $n\ge
1$ and $\pi((w^1))=\emptyset$, and we
say that $\pi(w)$
is the father of $w$. We set $\pi^0(w)=w$ and
$\pi^n(w)=\pi^{n-1}(\pi(w))$ for $1\leq n\leq |w|$. In particular,
$\pi^{|w|}(w)=\emptyset.$

A (finite or infinite) rooted ordered tree $\mathbf t$ is a subset of
$\cal W$ such that
\begin{enumerate}
\item $\emptyset\in {\mathbf t}.$

\item $w\in {\mathbf t}\setminus\{\emptyset\}\Longrightarrow\pi(w)\in {\mathbf t}.$

\item For every $w\in {\mathbf t}$, there exists a finite integer $k_{w}{\mathbf t}\geq
0$ such that, for every $j\in \mbb{N}$, $wj\in {\mathbf t}$ if and only if $0\leq
j\leq k_w\mathbf t$ ($k_w{\mathbf t}$ is the number of children of $w\in{\mathbf
t}$).
\end{enumerate}

Let ${\mathbf T}^{\infty}$ denote the set of all such trees $\mathbf t$.
Given a tree $\mathbf t$, we call an element in the set ${\mathbf
t}\subset{\cal W}$ a node of ${\mathbf t}$.  Denote the height of a tree
$\mathbf t$ by $|{\mathbf t}|:=\max\{|\nu|:\nu\in{\mathbf t}\}$. For $h\geq0$,
there exists a natural restriction map $r_h: {\mathbf T}^{\infty}\rar {\mathbf
T}^{h}$ such that $r_h{\mathbf t}=\{\nu\in{\mathbf t}:|\nu|\leq h\}$, where
${\mathbf T}^{h}:=\{\mathbf{t}\in \mathbf{T}^{\infty}: |{\mathbf
  t}|\leq h\}$. In particular, ${\mathbf
T}^0=\{\emptyset\}$.

We  denote by $\#{\mathbf t}$ the number of nodes of  ${\mathbf
t}$. Let
$$
{\mathbf T}:=\{{\mathbf t}\in{\mathbf T}^{\infty}: \#{\mathbf t}<\infty \}
$$
be the set of all finite trees. Then ${\mathbf
T}=\cup_{h=1}^{\infty}{\mathbf T}^h$.

We define the shifted subtree of
${\mathbf t}$ above $\nu$  by
$$
T_{\nu}{\mathbf t}:=\{w: \nu w\in{\mathbf t}\}.
$$

For $n\geq0$, let $\text{gen}(n, {\mathbf t})$ be the
$n$th generation of individuals in ${\mathbf t}$. That is
$$\text{gen}(n, {\mathbf t}):=\{\nu\in{\mathbf t}: |\nu|=n\}.
$$

 We say that $w\in{\mathbf t}$ is a
leaf of ${\mathbf t}$ if $k_w{\mathbf t}=0$ and set
$${\cal L}({\mathbf t}):=\{w\in{\mathbf t}: k_w{\mathbf t}=0\}.$$ So ${\cal L}({\mathbf t})$  denotes
the set of leaves of ${\mathbf t}$ and  $\#{\cal L}(\mathbf t)$ is the
number of leaves of ${\mathbf t}$.

We say that $w\in\mathbf{t}$ is an inner node of $\mathbf{t}$ if it is
not a leaf (i.e. $k_w\mathbf{t}>0$) and we denote by $\mathbf{t}^i$
the set of inner nodes of $\mathbf{t}$ i.e.
$$\mathbf{t}^i=\mathbf{t}\setminus\mathcal{L}(\mathbf{t}).$$

Given a probability distribution $p=\{p_n,\ n=0,1,\ldots\}$ with
$p_1<1$, following \cite{[AP98]}, call a random tree ${\cal G}_p$ a Galton-Watson tree with offspring distribution
$p$ if the number of children
 of $\emptyset$ has distribution $p$:
$$
\P(k_{\emptyset}{\cal G}_p=n)=p_n,\quad \forall\, n\geq0
$$
and for each $h=1,2,\ldots,$ conditionally given $r_h{\cal G}={\mathbf
t}^{h}\in {\mathbf T}^h$, for $\nu\in \mathrm{gen}(h,\mathbf{t}^h)$,
$k_{\nu}{\cal G}_p$ are i.i.d. random variables distributed according to $p$. That means
$$
\P(r_{h+1}{\cal G}={\mathbf t}\bigm|
r_h\mathcal{G}=r_h\mathbf{t})=\prod_{\nu\in r_h\mathbf{t}\setminus r_{h-1}\mathbf{t}}p_{{k_{\nu}{\mathbf
t}}},\quad {\mathbf t}\in {\mathbf T}^{h+1},
$$
where the product is over all nodes $\nu$ of $\mathbf t$ of height $h$. We
have then 
\bgeqn\label{ForGuT} 
\P({\cal G}={\mathbf
t})=\prod_{\nu\in{\mathbf t}}p_{k_{\nu}\mathbf{t}},\quad \mathbf{t}\in \mathbf{T},
\edeqn 
where the product is over all nodes $\nu$ of $\mathbf t$.

\subsection{Pruning at Nodes}
Let $\cal T$ be a tree in $\mathbf{T}^\infty$. For $0\leq u\leq 1$, a random tree
${\cal T}(u)$ is called a \textsl{node pruning} of ${\cal T}$ with
parameter $u$ if it
is constructed as follows: conditionally given ${\cal T}={\mathbf t}, \mathbf{t}\in
\mathbf{T}^{\infty}$, for $0\leq u\leq 1$, we consider a family of independent random
variables ($\xi^u_{\nu},\nu\in \mathbf{t})$ such that
$$P(\xi^u_{\nu}=1)=1-P(\xi^u_{\nu}=0)=\begin{cases}
 u^{k_{\nu}{\mathbf t}-1}, & \mbox{if }k_{\nu}{\mathbf t}\geq1,\\
 1, & \mbox{if }k_{\nu}{\mathbf t}=0,
 \end{cases}$$
and define \bgeqn\label{DefLu}{\cal
T}(u):=\{\emptyset\}\bigcup\left\{ \nu\in{\mathbf t}\setminus
\{\emptyset\}:\prod_{n=1}^{|\nu|}\xi_{\pi^n(\nu)}^u=1\right\}.\edeqn
This means that if a node $\nu$ belongs to ${\cal T}(u)$ and
$\xi_{\nu}^u=1$, then $\nu j, j=0,1,\ldots, k_{\nu}(\mathbf t)$ all
belong to ${\cal T}(u)$ and if $\xi_{\nu}^u=0$, then all  subsequent
offsprings of $\nu$ will be removed with the subtrees attached to
these nodes. Thus ${\cal T}(u)$ is a random
tree, ${\cal T}(u)\subset{\cal T}$ and we have for every $\nu\in \mathcal{W}$,
 \bgeqn\label{Foroffnu}
 \P(k_{\nu}{\cal T}(u)=n\bigm|\nu\in {\cal T}(u), \mathcal{T}=\mathbf{t})
=\begin{cases}
 u^{n-1}1_{\{k_{\nu}{\mathbf t}=n\}}, \quad &n\geq1,\\
 1_{\{k_{\nu}{\mathbf t}=0\}}+(1-u^{k_{\nu}{\mathbf t}-1})1_{\{k_{\nu}{\mathbf t}\geq1\}}, \quad &n=0.\end{cases}
  \edeqn
We also have that for $h\geq1$ and ${\mathbf t\in \mathbf{T}}^{\infty}$,
\begin{equation}
\label{Forsec} \P(r_{h}{\cal T}(u)=r_{h}{\mathbf t}\bigm|{\cal T}={\mathbf
t})=\P\left(\prod_{\nu\in n(h,{\mathbf
t})}\xi_{\nu}^u=1\right)=u^{\sum_{\nu\in
n(h, {\mathbf t})}(k_{\nu}{\mathbf t}-1)},
\end{equation}
where $n(h, {\mathbf
t}):=\{\nu\in{\mathbf t}: k_{\nu}{\mathbf t}\geq1 \text{ and }|\nu|<h\}$.

If  ${\cal T}$ is a Galton-Watson tree, we have the following proposition.
\begin{proposition}\label{PropGW}
If ${\cal T}$ is a Galton-Watson tree with offspring distribution $\{p_n,
n\geq0\}$, then ${\cal T}(u)$ is also a Galton-Watson tree with offspring
distribution $\{p_n^{(u)}, n\geq0\}$ defined by (\ref{eq:def_pu_1})
and (\ref{eq:def_pu_2}).
\end{proposition}
\begin{proof} By (\ref{Foroffnu}),
$$
\P(k_{\emptyset}{\cal T}(u)=0)=\P({\cal
T}=\{\emptyset\})+\sum_{n=1}^{\infty}\left(1-u^{n-1}\right)\P(k_{\emptyset}{\cal
T}=n)=p_0+\sum_{n=1}^{\infty}\left(1-u^{n-1}\right)p_n,
$$
which is equal to $p_0^{(u)}$. For $n\geq1$,
$$
\P(k_{\emptyset}{\cal T}(u)=n)=u^{n-1}\P( k_{\emptyset}{\cal
T}=n)=u^{n-1}p_n.
$$
The fact that $\{\xi_{\nu}^u\}$ are, conditionally on $\mathcal{T}$  independent random variables,
gives that for each $h=1,2,\ldots,$ conditionally given $r_h{\cal
T}(u)={\mathbf t}^{h}\in {\mathbf T}^h$, for $\nu\in{\mathbf t}^{h}$ with
$|\nu|=h$, $k_{\nu}{\cal T}(u)$ are independent. Meanwhile, again by
(\ref{Foroffnu}), 
\begin{multline*}
\P(k_{\nu}{\cal T}(u)=n\bigm|r_h{\cal
T}(u)={\mathbf t}^h)\\\qquad=
\begin{cases}
 u^{n-1}\P(k_{\nu}{\cal T}=n)=  p_n^{(u)},&\mbox{if }n\geq1,\\
 \P(k_{\nu}{\cal T }=0)+\sum_{k\geq1}(1-u^{k-1})P(k_{\nu}{\cal T
 }=k)=p_0^{(u)},  &\mbox{if }n=0.\end{cases}
\end{multline*}
Then the desired result follows readily. \qed
\end{proof}

\section{ A Tree-valued Markov Process}\label{sec:processes}

\subsection{A tree-valued process given the terminal tree}

Let $\cal T$ be a tree in $\mathbf{T}^\infty$. We want to construct a ${\mathbf T}^{\infty}$-valued
stochastic process $\{{\cal T}(u): 0\leq u\leq 1\}$ such that
\begin{itemize}
\item ${\cal T}(1)={\cal T}$,
\item for every $0\le u_1<u_2\le 1$, ${\cal T}(u_1)$ is a node pruning of ${\cal
  T}(u_2)$ with pruning parameter $u_1/u_2$.
\end{itemize}

Recall that ${\cal T}^i$ is the set of the inner nodes of $\cal T$. Let $(\xi_\nu, \nu\in{\cal
  T}^i)$ be a family of independent random variables such that, for
every $\nu\in{\cal G}^i$,
$$\P(\xi_\nu\le u)=u^{k_\nu{\cal T}-1}.$$

Then, for every $u\in[0,1]$, we set
$${\cal T}(u)=\left\{ \nu\in {\cal T},\ \forall 1\le n\le |\nu|,\
\xi_{\pi^n(\nu)}\le u\right\}.$$

We call the process $({\cal T}(u),0\le u\le 1)$ a pruning
process associated with $\cal T$. Let us remark that, contrary to the
process of \cite{[AP98]}, the tree ${\cal T}(0)$ may not be reduced to the root as the
nodes with one offspring are never pruned. More precisely, if we
denote by $(1)^n$ the $n$-uple $(1,1,\ldots,1)\in(\N^*)^n$ with the
convention $(1)^0=\emptyset$, we have
$$\mathcal{T}(0)=\bigl\{(1)^n,\ n\le\sup\{k,\ \forall l<k,\
  k_{(1)^l}\mathcal{G}=1\} \bigr\}$$
with the convention $\sup\emptyset=0$.

We deduce from Formula (\ref{Forsec}) the following proposition:
\begin{proposition}\label{Proplim}
We have that
 \bgeqn \label{Proplim1} \lim_{u\rar1}{\cal T}(u)={\cal T},\quad
a.s.,\edeqn where the limit means that for almost every $\omega$ in
the basic probability space, for each $h$ there exists a
$u(h,\omega)<1$ such  that $r_h {\cal T}(u,\omega)=r_h{\cal
T}(\omega)$ for all $u(h,\omega)<u\leq1$.
\end{proposition}

\subsection{Pruning Galton-Watson trees}\label{SsGW}

Let $p=\{p_n,\ n=0,1,\ldots\}$ be an offspring distribution. Let $\cal G$ be a Galton-Watson
tree with offspring distribution $p$. Then we consider the process
$({\cal G}(u),0\le u\le 1)$ such that, conditionally on $\cal G$,
the process is a pruning process associated with $\cal G$.

 Then for each $u\in [0,1]$,
 ${\cal G}(u)$ is a Galton-Watson tree with offspring distribution
$p^{(u)}$. Let $g(s)$ denote the generating function of $p$. Then the
distribution of ${\cal G}(u)$ is determined by the following
generating function
 \bgeqn\label{genu}
g_u(s)=1-g(u)/u+g(us)/u,\quad 0< u\leq 1.
 \edeqn

\subsection{Forward transition probabilities}

Let ${\cal L}(u)$ be the set of  leaves of ${\cal G}(u)$.  Fix
$\alpha$ and $\beta$ with $0\leq\alpha\leq\beta\leq 1$. Let us define
$$p_{\alpha,\beta}(k)=
\frac{(1-({\alpha}/{\beta})^{k-1})p^{(\beta)}_k}{p_0^{(\alpha)}}\quad
\text{for } k\geq1\quad \text{and }\quad
p_{\alpha,\beta}(0)=\frac{p_0^{(\beta)}}{p_0^{(\alpha)}}\cdot
$$

We define a modified Galton-Watson tree in which the size of the first
generation has distribution $p_{\alpha,\beta}$, while these
and all subsequent individuals have offspring distribution
$p^{(\beta)}$. More precisely, let $N$ be a random variable with law
$p_{\alpha,\beta}$ and let $(\mathcal{T}_k,k\in\N^*)$ be a sequence of
i.i.d. Galton-Watson trees with offspring distribution $p^{(\beta)}$
independent of $N$. Then we define the modified Galton-Watson tree
$\mathcal{G}_{\alpha,\beta}$ by
\begin{equation}\label{eq:def_modified}
\mathcal{G}_{\alpha,\beta}=\{\emptyset\}\cup \{k,1\le k\le N\}\cup
\bigcup_{k=1}^N\{kw,\ w\in \mathcal{T}_k\}.
\end{equation}

Let $(\mathcal{G}_{\alpha,\beta}^\nu,\nu\in\mathcal{L}(\alpha))$ be,
conditionally given $\mathcal{G}(\alpha)$, i.i.d. modified
Galton-Watson trees. We set
\begin{equation}\label{eq:defghat}
\hat{{\cal G}}({\beta})={\cal G}({\alpha})\cup\bigcup_{\nu\in{\cal
L}(\alpha)}\{\nu w:w\in {\cal G}_{\alpha, \beta}^{\nu} \}.
\end{equation}
That is $\hat{{\cal G}}({\beta})$ is a random tree obtained by
 adding a modified Galton-Watson tree ${\cal G}_{\alpha,\beta}^{\nu}$
  on each leaf $\nu$ of ${\cal G}(\alpha)$.
The following proposition, which implies the Markov property of
$\{{\cal G}(u),\  0\leq u\leq 1\}$, describes the transition
 probabilities of that tree-valued process.

\begin{proposition}\label{PropMK}
For every $0\le\alpha\le\beta\le 1$, $({\cal G}(\alpha),
{\cal G}(\beta))\overset{d}{=}({\cal G}(\alpha), \hat{{\cal
G}}(\beta))$.
\end{proposition}

\begin{proof} 
Let $\alpha<\beta$, let $h\in\N^*$ and let $\mathbf{s}$ and $\mathbf{t}$
be two trees of $\mathbf{T}^h$
such that $\mathbf{s}$ can be obtained from $\mathbf{t}$ by pruning
at nodes. Then, by definition of the pruning procedure, we have
\begin{align*}
\P(r_h\mathcal{G}(\alpha)=\mathbf{s},r_h\mathcal{G}(\beta)=\mathbf{t}) &
=\prod_{\nu\in r_{h-1}\mathbf{t}}p_{k_\nu\mathbf{t}}^{(\beta)}\prod_{\nu\in\mathbf{s}^i}\left(\frac{\alpha}{\beta}\right)^{k_\nu\mathbf{t}-1}\prod_{\nu\in\mathcal{L}(\mathbf{s})\setminus\mathcal{L}(\mathbf{t})}\left(1-\left(\frac{\alpha}{\beta}\right)^{k_\nu\mathbf{t}-1}\right)\\
&=\prod_{\nu\in\mathbf{s}^i}p_{k_\nu\mathbf{t}}^{(\alpha)}\prod_{\nu\in
  r_{h-1}\mathbf{t}\setminus\mathbf{s}^i}p_{k_\nu\mathbf{t}}^{(\beta)}\prod_{\nu\in\mathcal{L}(\mathbf{s}),
|\nu|<h}\left(1-\left(\frac{\alpha}{\beta}\right)^{k_\nu\mathbf{t}-1}\mathbf{1}_{\{k_\nu\mathbf{t}>0\}}\right)\\
&=\prod_{\nu\in\mathbf{s}^i}p_{k_\nu\mathbf{t}}^{(\alpha)}\prod_{\nu\in
  r_{h-1}\mathbf{t}\setminus\mathbf{s}}p_{k_\nu\mathbf{t}}^{(\beta)}\prod_{\nu\in\mathcal{L}(\mathbf{s}),|\nu|<h}p_{k_\nu\mathbf{t}}^{(\beta)}\left(1-\left(\frac{\alpha}{\beta}\right)^{k_\nu\mathbf{t}-1}\mathbf{1}_{\{k_\nu\mathbf{t}>0\}}\right)\\
&=\prod_{\nu\in r_{h-1}\mathbf{s}}p_{k_\nu\mathbf{s}}^{(\alpha)}\prod_{\nu\in r_{h-1}\mathbf{t}\setminus\mathbf{s}}p_{k_\nu\mathbf{t}}^{(\beta)}\prod_{\nu\in\mathcal{L}(\mathbf{s}),|\nu|<h}\frac{p_{k_\nu\mathbf{t}}^{(\beta)}}{p_{0}^{(\alpha)}}\left(1-\left(\frac{\alpha}{\beta}\right)^{k_\nu\mathbf{t}-1}\mathbf{1}_{\{k_\nu\mathbf{t}>0\}}\right)\\
&=\prod_{\nu\in
  r_{h-1}\mathbf{s}}p_{k_\nu\mathbf{s}}^{(\alpha)}\prod_{\nu\in r_{h-1}\mathbf{t}\setminus\mathbf{s}}p_{k_\nu\mathbf{t}}^{(\beta)}\prod_{\nu\in\mathcal{L}(\mathbf{s}),|\nu|<h}p_{\alpha,\beta}(k_\nu\mathbf{t}).
\end{align*}
The definition of $\hat{\mathcal G}(\beta)$  readily
implies
$$\P(r_h\mathcal{G}(\alpha)=\mathbf{s},r_h\hat{\mathcal{G}}(\beta)=\mathbf{t})=\prod_{\nu\in
  r_{h-1}\mathbf{s}}p_{k_\nu\mathbf{s}}^{(\alpha)}\prod_{\nu\in r_{h-1}\mathbf{t}\setminus\mathbf{s}}p_{k_\nu\mathbf{t}}^{(\beta)}\prod_{\nu\in\mathcal{L}(\mathbf{s}),|\nu|<h}p_{\alpha,\beta}(k_\nu\mathbf{t})$$
which ends the proof.
\qed
\end{proof}

 Let $\#{\cal L}(u)$ denote the number of leaves of ${\cal
G}(u)$. 
The latter proposition together with the description of $\hat{\cal G}$  readily imply
  \bgeqn
 \label{PropMK1}
({\cal G}(\alpha), \#{\cal L}(\beta))\overset{d}{=}\left({\cal
G}(\alpha), \sum_{\nu\in{\cal L}(\alpha)}\#{\cal L}({\cal
G}_{\alpha,\beta}^{\nu})\right).
 \edeqn

We can also describe the forward transition rates when the trees are finite. If $\mathbf{s}$
and $\mathbf{t}$ are two trees and if $\nu\in\mathcal{L}(\mathbf{s})$,
we define the tree obtained by grafting $\mathbf{t}$ on $\nu$ by
$$\mathbf{r}(\mathbf{s},\nu;\mathbf{t}):=\mathbf{s}\cup\{\nu w,
w\in\mathbf{t}\}.$$

We also define, for $s\in \mathbf{T}$, $\nu\in\mathcal{L}(\mathbf{s})$
and $k\in\N^*$, the set of trees
$$\mathbf{r}(\mathbf{s},\nu;\mathbf{t}_k(\infty)):=\{\mathbf{r}(\mathbf{s},\nu;\mathbf{t}),\ k_\emptyset\mathbf{t}=k,\ \#\mathbf{t}=\infty\}$$

\begin{corollary}
Let $\mathbf{s}\in\mathbf{T}$, $\mathbf{t}\in
\mathbf{T}$, $\mathbf{t}\ne \{\emptyset\}$ and let
$\nu\in\mathcal{L}(\mathbf{s})$. Then the transition rate at time $u$
from $\mathbf{s}$ to ${\bf r}({\bf s},\nu;{\bf t})$ is
  given by
\begin{equation}\label{fintran}
q_{u}({\bf
s}\rar{\bf r}({\bf s},\nu;{\bf t})):=\frac{k_\emptyset\mathbf{t}-1}{u}\frac{\P({\cal
G}(u)={\bf t})}{p_0^{(u)}}.
\end{equation}
Let $\mathbf{s}\in\mathbf{T}$,
$\nu\in\mathcal{L}(\mathbf{s})$ and $k \ge 1$. Then the transition rate at time $u$
from $\mathbf{s}$ to the set
$\mathbf{r}(\mathbf{s},\nu;\mathbf{t}_k(\infty))$ is
\begin{equation}\label{fintran2}
q_{u}({\bf
s}\rar\mathbf{r}(\mathbf{s},\nu;\mathbf{t}_k(\infty)):=\frac{k-1}{u}\frac{1-F(u)^k}{p_0^{(u)}}p_k^{(u)},
\end {equation}
and no other transitions are allowed.
\end{corollary}

\begin{proof} Let $\mathbf{s}\in\mathbf{T}$, $\mathbf{t}\in
\mathbf{T}$, $\mathbf{t}\ne \{\emptyset\}$ and let
$\nu\in\mathcal{L}(\mathbf{s})$.
By Proposition \ref{PropMK}, we have
\begin{align*}
\P\left(\mathcal{G}(u)=\mathbf{s},\mathcal{G}(u+du)=\mathbf{r}(\mathbf{s},\nu;\mathbf{t})\right)
&
=\P\left(\mathcal{G}(u)=\mathbf{s},\hat{\mathcal{G}}(u+du)=\mathbf{r}(\mathbf{s},\nu;\mathbf{t})\right)\\
&
=\P\left(\mathcal{G}(u)=\mathbf{s}\right)\P(\mathcal{G}_{u,u+du}^\nu=\mathbf{t})\prod_{\tilde\nu\in\mathcal{L}(\mathbf{s})\setminus\{\nu\}}\P(\mathcal{G}_{u,u+du}^{\tilde\nu}=\{\emptyset\})\\
&=\P\left(\mathcal{G}(u)=\mathbf{s}\right)\P(\mathcal{G}_{u,u+du}^\nu=\mathbf{t})p_{u,u+du}(0)^{\#\mathcal{L}(\mathbf{s})-1}.
\end{align*}

Using (\ref{eq:def_modified}), we get
\begin{align*}
\P 
\Bigl(\mathcal{G}(u+du)=\mathbf{r}(\mathbf{s},\nu;\mathbf{t}) & \bigm|\mathcal{G}(u)=\mathbf{s}\Bigr)\\
&
=\P\left(\mathcal{G}(u+du)=\mathbf{t}\right)\frac{p_{u,u+du}(k_\emptyset\mathbf{t})}{p_{k_\emptyset\mathbf{t}}^{(u+du)}}p_{u,u+du}(0)^{\#\mathcal{L}(\mathbf{s})-1}\\
&
=\P\left(\mathcal{G}(u+du)=\mathbf{t}\right)\frac{1}{p_0^{(u)}}\left(1-\left(\frac{u}{u+du}\right)^{k_\emptyset\mathbf{t}-1}\right)\left(\frac{p_0^{(u+du)}}{p_0^{(u)}}\right)^{\#\mathcal{L}(\mathbf{s})-1}\\
&\underset{du\to
  0}{\sim}\P\left(\mathcal{G}(u)=\mathbf{t}\right)\frac{1}{p_0^{(u)}}(k_\emptyset\mathbf{t}-1)\frac{du}{u}\cdot
\end{align*}

This gives Formula (\ref{fintran}). 

A similar computation replacing $\P(\mathcal{G}(u)=\mathbf{t})$ by
$\P(\#\mathcal{G}(u)=\infty\bigm|k_\emptyset\mathcal{G}(u)=k)=1-F(u)^k$
gives Formula (\ref{fintran2}).

Another similar computation gives that,
if $\mathbf{t}$ is obtained by grafting two trees (or more) on the
leaves of $\mathbf{s}$,
$$\P\left(\mathcal{G}(u+du)=\mathbf{t}\bigm|\mathcal{G}(u)=\mathbf{s}\right)=o(du)$$
and it is clear by construction that, in all the other cases,
$$\P\left(\mathcal{G}(u+du)=\mathbf{t}\bigm|\mathcal{G}(u)=\mathbf{s}\right)=0.$$
\qed
 \end{proof}

Let us define 
\begin{equation}\label{eq:defmu}
\mu(u):=\sum_{k=1}^\infty kp_k^{(u)}
\end{equation}
if it exists the mean of
$p^{(u)}$. We set $u_1=\sup\{u\in[0,1],\ \mu(u)\leq 1\}$. Recall the
definition of function $M$ in (\ref{eq:defh}).

\begin{corollary}\label{CorL} The process
\bgeqn\label{CorL2} \left(M(u,\mathcal{G}(u)),\quad
0\leq u<u_1\right), \edeqn 
is a martingale with respect to the filtration
generated by $\{{\cal G}(u),\ 0\leq u< u_1\}$.
\end{corollary}

\begin{proof}
First, by the branching property of Galton-Watson process, for each
$n\geq1$, $0\leq u<u_1$ and $\ell\ge n$,
$$
\P\left(\#{\cal L}(u)=\ell\bigm|k_{\emptyset}{\cal
G}(u)=n\right)=\P\left(\sum_{i=1}^{n}L_i=\ell\right),
$$
where $L_1,L_2,\ldots$ are i.i.d. copies of $\#{\cal L}(u)$. This
gives $$\E\left[\#{\cal L}(u)\right]=p_0^{(u)}+\E\left[k_{\emptyset}{\cal G}(u)\right]\E\left[\#{\cal
L}(u)\right] $$ which implies
 \bgeqn\label{CorL1} \E\left[\#{\cal
     L}(u)\right]=\frac{p_0^{(u)}}{1-\mu(u)}\cdot \edeqn
A straightforward computation gives that the mean of the offspring
distribution $p_{\alpha,\beta}$ is
$$\mu_{\alpha,\beta}:=\frac{\mu(\beta)-\mu(\alpha)}{p_0^{(\alpha)}}.$$
By the same reasoning, (\ref{PropMK1}) and (\ref{eq:def_modified}) imply,
for $0\leq\alpha\leq\beta<u_1$,
\begin{align*}
\E\left[\#{\cal L}(\beta)\bigm|{\cal G}(\alpha)\right]=\#{\cal L}(\alpha)\E\left[\#{\cal
L}({\cal G}_{\alpha,\beta}^{\nu})\right] & =\#{\cal
L}(\alpha)\left(p_{\alpha,\beta}(0)+\mu_{\alpha,\beta}\E\left[\#{\cal
L}(\beta)\right]\right)\\
& =\#{\cal L}(\alpha)\left(p_{\alpha,\beta}(0)+\mu_{\alpha,\beta}\frac{p_0^{(\beta)}}{1-\mu(\beta)}\right)
\end{align*}
by (\ref{CorL1}).
Then the martingale property of (\ref{CorL2})
follows from a simple calculation. \qed

\end{proof}

\bigskip

\subsection{Pruning a Galton-Watson Tree Conditioned on Non-Extinction}

Let $p$ be a critical or sub-critical offspring distribution with mean
$\mu$ such that $p_0<1$.
We define the size-biased probability distribution $p^*$ of
$p$ by
$$
p^*_k=\frac{kp_k}{\mu},\quad k\geq0.
$$

Let $\cal G$ be a Galton-Watson tree  with
offspring distribution $p$. For a tree $\mathbf{t}$, we denote by $Z_n\mathbf{t}=\#\mathrm{gen}(n,\mathbf{t})$
the number of individuals in the $n$th generation of $\mathbf{t}$.
 We first recall a result in \cite{[K87]}.

\begin{proposition}\label{Propkesten}
(Kesten \cite{[K87]}, Aldous and Pitman \cite{[AP98]})

(i) The conditional distribution of $\mathcal{G}$ given
  $\{Z_n\mathcal{G}>0\}$ converges, as $n$ tends to $+\infty$,  toward the law of a random family tree ${\cal G}^{\infty}$ specified
by
$$
\P(r_h{\cal G}^{\infty}={\mathbf t})=\mu^{-h}(Z_h{\mathbf t})\P(r_h{\cal
G}={\mathbf t})\quad \forall\, {\mathbf t\in \mathbf{T}}^{(h)}, h\geq0.
$$

(ii) Almost surely ${\cal G}^{\infty}$ contains a unique infinite
path ($\emptyset=V_0, V_1, V_2,\ldots$) such that $\pi(V_{h+1})=V_h$ for every $h=0,1,2,\ldots$.

(iii) The joint distribution of $(V_0, V_1,
V_2,\ldots)$ and ${\cal G}^{\infty}$ is determined recursively as
follows: for each $h=0,1,2,\ldots,$ given $(V_0, V_1, V_2,\ldots,
V_h)$ and $r_h{\cal G}^{\infty}$, the numbers of children
$(k_{\nu}{\cal G}^{\infty},\nu\in \text{gen}(h, {\cal G}^{\infty}))$
are independent  with distribution $p$ for
$\nu\neq V_h$, and with the size-biased distribution $p^*$
for $\nu=V_h$; given also the numbers of children $k_{\nu}{\cal
G}^{\infty}$ for $\nu\in\text{gen}(h, {\cal G}^{\infty})$, the
vertex $V_{h+1}$ has uniform distribution on the set $\{(V_h,i), 1\le
i\le k_{V_h}{\cal G}^{\infty}\}$.
\end{proposition}

We say that $\mathcal{G}^\infty$ is the Galton-Watson tree associated
with $p$ conditioned on non-extinction. We then define the process 
$({\cal G}^{*}(u), 0\leq u\leq 1)$ as a pruning process
  associated with ${\cal
G}^{\infty}$. 
By Proposition \ref{Proplim}, ${\cal G}^{*}(1-)={\cal
G}^{*}(1)={\cal G}^{\infty}(1)$ almost surely. And since there
exists a unique infinite path, we get that ${\cal G}^{*}(u)$ is finite
almost
surely for all $0\leq u<1$.

The distribution of ${\cal G}^*(u)$ for fixed $u$
is given in the following proposition. Let us recall that $\mu(u)$ is
the mean of $p^{(u)}$ defined in (\ref{eq:defmu}).

\begin{proposition} \label{PropGstar}
For each $0\leq u<1$,
\bgeqn\label{PropGstar1}\P({\cal G}^{*}(u)={\mathbf
t})=\left(\sum_{\nu\in {\cal L}({\mathbf
t})}\frac{1}{\mu(1)^{|\nu|+1}}\right)\frac{\mu(1)-\mu(u)}{p_0^{(u)}}\P({\cal
G}(u)={\mathbf t}),\quad {\mathbf t\in \mathbf{T}}.\edeqn
\end{proposition}

\begin{proof} We prove (\ref{PropGstar1}) inductively.
First, note that 
$$\P(k_{\emptyset}{\cal
G}^{\infty}(1)=n)=p_n^*=np_n/\mu(1).$$
Then
$$
\P({\cal
G}^{*}(u)=\{\emptyset\})=\sum_{n\geq1}(1-u^{n-1})\P(k_{\emptyset}{\cal
G}^{\infty}(1)=n)=(\mu(1)-\mu(u))/\mu(1).
$$
Since $\P({\cal G}(u)=\{\emptyset\})=p_0^{(u)}$, 
(\ref{PropGstar1}) holds for ${\mathbf t}=\{\emptyset\}$.

On the other
hand, by Proposition \ref{Propkesten}, we have
$$
\P\left(T_{\nu}{\cal
G}^{\infty}(1)=\mathbf{t}\bigm|\nu=V_{|\nu|}\right)=\P\left({\cal G}^{\infty}=\mathbf{t}\right)
$$
 and
 $$ \P\left(T_{\nu}{\cal
G}^{\infty}(1)=\mathbf{t}\bigm|\nu\neq V_{|\nu|}\right)=\P\left({\cal G}=\mathbf{t}\right)$$ which
gives
 \bgeqn\label{PropGstar2}
\P\left(T_{\nu}{\cal G}^*(u)=\mathbf{t}\bigm|\nu\in {\cal
G}^*(u),\nu=V_{|\nu|}\right)=\P\left({\cal G}^*(u)=\mathbf{t}\right) \edeqn and
 \bgeqn\label{PropGstar3}
\P\left(T_{\nu}{\cal G}^*(u)=\mathbf{t}\bigm|\nu\in {\cal G}^*(u),\nu\neq
V_{|\nu|}\right)=\P\left({\cal G}(u)=\mathbf{t}\right) ,\edeqn 
respectively. Meanwhile,
since ${\cal G}(u)$ is a Galton-Watson tree,
$$
\P\left(T_{\nu}{\cal G}(u)=\mathbf{t}\bigm|\nu\in {\cal G}(u)\right)=\P({\cal
G}(u)=\mathbf{t}),$$
which implies
 \bgeqn\label{PropGstar4}
\P({\cal G}(u)={\mathbf t})=p_{k_{\emptyset}{\mathbf t}}^{(u)}\prod_{1\leq j\leq
k_{\emptyset}{\mathbf t}}\P({\cal G}(u)=T_{(j)}{\mathbf t}). \edeqn 

For some
$h\geq0$, assume that (\ref{PropGstar1}) holds for all trees in
${\mathbf T}^h$. By (\ref{PropGstar2}) and
(\ref{PropGstar3}), we have for ${\mathbf{t}\in \mathbf{T}}^{h+1}\backslash{\mathbf
T}^h$,
 \begin{align*}
\P({\cal G}^{*}(u)={\mathbf t}) &=
\P(k_\emptyset\mathcal{G}^*(u)=k_\emptyset\mathbf{t})\sum_{i=1}^{k_\emptyset\mathbf{t}}\P(\forall
1\le j\le k_\emptyset\mathbf{t},\ T_{(j)}\mathcal{G}^*(u)=T_{(j)}\mathbf{t}\bigm|V_1=i)\P(V_1=i)\\
&=u^{k_{\emptyset}{\mathbf
t}-1}p^*_{k_{\emptyset}{\mathbf t}}\frac{1}{k_{\emptyset}{\mathbf
t}}\sum_{i=1}^{k_{\emptyset}{\mathbf t}}\left(\P({\cal
G}^{*}(u)=T_{(i)}{\mathbf t})\cdot\prod_{j\neq i,1\leq j\leq
k_{\emptyset}{\mathbf
t}}\P\left({\cal G}(u)=T_{(j)}{\mathbf t}\right)\right)\\
&=\frac{p_{k_{\emptyset}{\mathbf
t}}^{(u)}}{\mu(1)}\sum_{i=1}^{k_{\emptyset}{\mathbf t}}\sum_{\nu\in {\cal
L}(T_{(i)}{\mathbf
t})}\left(\frac{1}{\mu(1)^{|\nu|}}\frac{\mu(1)-\mu(u)}{p_0^{(u)}}\cdot\prod_{1\leq
j\leq k_{\emptyset}{\mathbf t}}\P\left({\cal G}(u)=T_{(j)}{\mathbf
t}\right)\right)\\
 &=\left(\sum_{\nu'\in {\cal L}({\mathbf
t})}\frac{1}{\mu(1)^{|\nu'|+1}}\right)\frac{\mu(1)-\mu(u)}{p_0^{(u)}}\P({\cal
G}(u)={\mathbf t}),
\end{align*}
 where the last equality follows from
(\ref{PropGstar4}) and the facts
$$\bigcup_{1\leq i\leq k_{\emptyset}{\mathbf t}}\{ i\nu: \nu\in {\cal
L}(T_{(i)}{\mathbf t})\}={\cal L}(\mathbf t)$$ and
 $|i\nu|=|\nu|+1$. Since ${\mathbf
T}=\cup_{h=1}^{\infty}{\mathbf T}^h$, (\ref{PropGstar1}) follows
inductively. \qed

\end{proof}
\begin{remark}\label{rem:Gstar}
If ${\cal G}$ is a critical Galton-Watson tree ($\mu(1)=1$) with $p_1<1$, Formula
(\ref{PropGstar1}) becomes
$$\P({\cal G}^{*}(u)={\mathbf t})=\frac{\#{\cal L}({\mathbf
t})(1-\mu(u))}{p_0^{(u)}}\P({\cal G}(u)={\mathbf t})=M(u,\mathbf{t})\P({\cal G}(u)={\mathbf t}).$$
In other words, the law of $\mathcal{G}^*(u)$ is absolutely continuous
with respect to the law of $\mathcal{G}(u)$ with density the
martingale of Corollary \ref{CorL}.
\end{remark}

\section{The Ascension Process and Its
  Representation}\label{sec:representation}
In this section, we consider a critical offspring distribution $p$
with $p_1<1$
and set
$$I=\left\{u\ge 0,\ \sum_{k=1}^{+\infty}u^{k-1}p_k\le 1\right\}.$$

\begin{remark}
We have $I=[0,\bar u]$  with $\bar u\ge 1$. Indeed, as the all the
coefficients of the sum are nonnegative, either the sum converges at
the radius of convergence $R$ and is continuous on $[0,R]$, or it
tends to infinity when $u\to R$. In the latter case, there is a unique
$\bar u<R$ for which the generating function takes the value 1 (by
continuity). In the former case, either the value of the generating
function at $R$ is greater than 1 and the previous argument also
applies, or the value of the generating function at $R$ is less than 1
and $\bar u=R$.
\end{remark}

Let us give some examples:

\begin{example} The binary case.\\
We consider the critical offspring distribution $p$ defined by $p_0=p_2=1/2$
(each individual dies out or gives birth to two children with equal
probability). In that case, we have
$$\sum_{k=1}^{+\infty} u^{k-1}p_k=\frac{1}{2}u$$
and hence we have $\bar u=2$, $I=[0,2]$ and $p_0^{(2)}=0$.
\end{example}

\begin{example} The geometric case.\\
We consider the critical offspring distribution $p$ defined by
$$\begin{cases}
p_k=\alpha \beta^{k-1} & \mbox{for }k\ge 1,\\
p_0=1-\frac{\alpha}{1-\beta}\cdot
\end{cases}$$
Then, for every $u$, $p^{(u)}$ is still of that form. As the offspring
distribution $p$ is critical, we must have $\alpha=(1-\beta)^2$,
$0<\beta<1$. In that case, we have
$$\sum_{k=1}^{+\infty}p_ku^{k-1}=\frac{(1-\beta)^2}{1-\beta u}$$
and hence we have $\bar u=2-\beta$, $I=[0,\bar u]$, $p_0^{(\bar u)}=0$.
\end{example}

\begin{example} We consider the critical offspring distribution $p$ defined by
$$\begin{cases}
\displaystyle p_k=\frac{6}{\pi^2}\frac{1}{k^3} & \mbox{for }k\ge 1,\\
\displaystyle p_0=1-\sum_{k=1}^{+\infty}p_k,
\end{cases}$$
then $\bar u=1$ and $p_0^{(\bar u)}=p_0>0$.
\end{example}

For $u\in I$, let us define
 \bgeqn
 \label{ForPu1}
 \left\{\begin{array}{lll}
 p_k^{(u)}&=&u^{k-1}p_k,\quad k\geq1,\\
 p_0^{(u)}&=&\displaystyle 1-\sum_{k=1}^{+\infty}p_k^{(u)}.
 \end{array}\right.
 \edeqn
Then, for $u\in I$, $p^{(u)}$ is still an offspring distribution, it
is sub-critical for $u<1$, critical for $u=1$ and super-critical for $u>1$.

We consider a tree-valued process $(\mathcal{G}(u),\ u\in I)$ such
that the process $(\mathcal{G}(t\bar u),\ t\in [0,1])$ is a pruning process
associated with $\mathcal{G}(\bar u)$. Then this process $\mathcal{G}$
satisfies the following properties:
\begin{itemize}
\item for every $u\in I$, $\mathcal{G}(u)$ is a Galton-Watson tree
  with offspring distribution $p^{(u)}$,
\item for every $\alpha,\beta\in I$, $\alpha<\beta$,
  $\mathcal{G}(\alpha)$ is a pruning of $\mathcal{G}(\beta)$.
\end{itemize}

We now consider
$\{{\cal G}(u),\ u\in I\}$ as an \textsl{ascension
process} with the \textsl{ascension time}
$$A:=\inf\{u\in I,\ \# {\cal G}(u)=\infty\}$$
with the convention $\inf\emptyset=\bar u$.

The state in the ascension process at time $u$  is ${\cal G}(u)$ if
$0\leq u< A$ and ${\mathbf t(\infty)}$ if $A\leq u$ where ${\mathbf
t(\infty)}$ is a state representing any infinite tree. Then the
ascension process is still a Markov process with countable
state-space ${\mathbf T}\cup\{{\mathbf t}(\infty)\}$, where $\mathbf t(\infty)$ is an absorbing
state.

Denote by $F(u)$ the extinction probability of a Galton-Watson process with offspring distribution $p^{(u)}$, which is the least
non-negative root of the following equation with respect to $s$
\bgeqn\label{Forgu} s=g_u(s)=1-g_1(u)/u+g_1(us)/u \edeqn
where $g_1$ is the generating function associated with the offspring
distribution $p$ and $g_u$ is the generating function associated with $p^{(u)}$.

We set
\begin{equation}\label{eq:defF}
\bar F(u)=1-{F}(u).
\end{equation}

Thus
\begin{equation}\label{eq:dual}
u\bar{F}(u)+g_1(u-u\bar{F}(u))=g_1(u).
\end{equation}

The distribution of the ascension process is determined by the
transition rates
(\ref{fintran}) and
 \bgeqn\label{Forratein}
q_{u}({\bf s}\rar{\bf t (\infty)})=\frac{\#{\cal L}({\bf s})}{up_0^{(u)}}
\sum_{k=2}^{\infty}(k-1)p_k^{(u)}(1-F(u)^k).
 \edeqn

 Define the conjugate
  $\hat{u}$ by 
\begin{equation}\label{eq:defhatu}
\hat{u}=uF(u)\qquad \mbox{for }u\in I.
\end{equation}

In particular, for $u\le 1$, $F(u)=1$ and consequently $\hat u=u$. On
the contrary, for $u>1$, Proposition  \ref{PropGuGuhat} shows that
$\hat u\le 1$.

We can restate Equation (\ref{eq:dual}) into
\begin{equation}\label{eq:dual2}
g_1(\hat u)-g_1(u)=\hat u-u.
\end{equation}
Notice that this equation with the condition $\hat u\le 1$
characterizes $\hat u$.

We first prove the following result which is already well-known, see
for instance \cite{[AN72]}, p52. We just restate this property in terms of
our pruning parameter. 

\begin{proposition}
 \label{PropGuGuhat}For any $u\in I$, $u\ge 1$
 \bgeqn\label{PropGuGuhat1}
\P({\cal G}(u)={\mathbf t})=F(u)\P({\cal G}(\hat{u})={\mathbf t}),\quad {\mathbf
t\in \mathbf{T}}.
 \edeqn
\end{proposition}

In other words, the law of $\mathcal{G}(u)$ conditioned to be finite
is $\mathcal{G}(\hat u)$, which explains the term {\it conjugate} for
$\hat u$. 

\begin{proof}
By (\ref{ForGuT}), we have
$$\P({\cal G}(u)={\mathbf t})=\prod_{\nu\in {\mathbf t}\setminus{\cal L}({\mathbf t})}p_{k_{\nu}{\mathbf t}}^{(u)}\cdot \prod_{\nu\in{\cal L}({\mathbf t})}p_0^{(u)} $$
and by (\ref{eq:defhatu}),
$$
\P({\cal G}(u)={\mathbf t})=F(u)^{-\left(\sum_{\nu\in {\mathbf t}\setminus{\cal
L}({\mathbf t})}(k_{\nu}{\mathbf
t}-1)\right)}\left(\frac{p_0^{(u)}}{p_0^{(\hat{u})}}\right)^{\#{\cal L}({\mathbf
t})}\P({\cal G}(\hat{u})={\mathbf t}).
$$

We also have
\begin{equation}\label{eq:pzero}
p_0^{(\hat{u})}=1-\sum_{k=1}^{\infty}F(u)^{k-1}p_k^{(u)}=1+p_0^{(u)}/F(u)-g_u(F(u))/F(u)=p_0^{(u)}/F(u).
\end{equation}

Then the desired result follows from the fact that given a tree
$\mathbf{t}\in \mathbf{T}$, 
$$\#\mathcal{L}({\mathbf t})=1+\sum_{\nu\in {\mathbf t}\setminus{\cal
L}({\mathbf t})}(k_{\nu}{\mathbf t}-1).$$\qed
\end{proof}

In what follows, we will often suppose that
\begin{equation}\label{assump}
p_0^{(\bar u)}=0,
\end{equation}
which is equivalent to the condition
$$\sum_{k=1}^{+\infty}\bar u^{k-1}p_k=1$$
and which implies that $\mathcal{G}(\bar u)$ is infinite.

We can however give the law of $A$ and of the tree
$\mathcal{G}(A-)$ just before the ascension time in general, this is
the purpose of the next proposition.

\begin{proposition}\label{PropA}
 For $u\in[1,\bar{u})$ and $\mathbf{t}\in \mathbf{T}$,
 \bgeqn\label{PropA1}\P(A\leq u)=\bar{F}(u).\edeqn
\bgeqn\label{PropA3} \P({\cal G}(A-)={\mathbf t}| A=u)=M(\hat u,\mathbf{t})\P({\cal G}(\hat{u})={\mathbf t}).
\edeqn
\bgeqn\label{PropA5}\P(\#\mathcal{G}(A-)<+\infty\bigm|A=u)=1.\edeqn

Furthermore, under assumption (\ref{assump}),
\bgeqn\label{PropA1bis}\P(A<\bar u)=1\edeqn
and
 \bgeqn\label{PropA4}
 \left(A, \frac{\hat{A}}{A}\right)=(A, F(A))
 \overset{d}{=}\left( \bar{F}^{-1}(1-\gamma),\gamma\right),
\edeqn where $\bar{F}^{-1}: [0,1]\rar[1,\bar{ u}]$ is the inverse
function of $\bar{F}$ and $\gamma$ is a r.v uniformly distributed on $(0,1)$.
\end{proposition}

\begin{proof}
We have
$$\P(A\leq u)=\P(\# {\cal
 G}(u)=\infty)=\bar{F}(u)$$
which gives (\ref{PropA1}). 

By the definition of $F(u)$ in Formula (\ref{Forgu}) and the implicit
function Theorem, function $F$ is differentiable on $(1,\bar u)$. 
This gives for $u\in (1,\bar u)$
$$\P(A\in du)=-F'(u)du$$
and differentiating (\ref{Forgu}) gives
\begin{equation}\label{eq:expr_F'}
u\bigl(1-g_1'(\hat u)\bigr)F'(u)=1-g'_1(u)-F(u)\bigl(1-g'_1(\hat u)\bigr).
\end{equation}

By (\ref{Forratein}), we have for $\mathbf{t}\in\mathbf{T}$
\bgeqn\label{PropA2} \P({\cal G}(A-)={\mathbf t}, A\in du)=\frac{\#{\cal
L}({\mathbf t})\P({\cal G}(u)={\mathbf t})}{u p_0^{(u)}}
\sum_{k=2}^{\infty}{(k-1)}{p_k^{(u)}}(1-F(u)^k)du. \edeqn

Now, using (\ref{PropGuGuhat1}), we have
$$\P({\cal G}(A-)={\mathbf t}, A\in du)=\frac{\#{\cal
L}({\mathbf t})\P({\cal G}(\hat u)={\mathbf t})F(u)}{u p_0^{(u)}}
\sum_{k=2}^{\infty}{(k-1)}{p_k^{(u)}}(1-F(u)^k)du.$$

Easy computations give
\begin{align*}
\sum_{k=2}^{\infty}{(k-1)}{p_k^{(u)}}(1-F(u)^k)
& =g_1'(u)-\frac{g_1(u)}{u}-F(u)g'_1(\hat u)+\frac{g(\hat u)}{u}\\
& =-F'(u)u\bigl(1-g'_1(\hat u)\bigr)+1-F(u)+\frac{g_1(\hat
  u)-g_1(u)}{u}\\
& =-F'(u)u\bigl(1-g'_1(\hat u)\bigr)
\end{align*}
using first Equation (\ref{eq:expr_F'}) and then Equation
(\ref{eq:dual2}).

This finally gives
\begin{align*}
\P({\cal G}(A-)={\mathbf t}, A\in du) & =-\frac{\#{\cal
L}({\mathbf t})\P({\cal G}(\hat u)={\mathbf
    t})F(u)}{p_0^{(u)}}F'(u)\bigl(1-g'_1(\hat u)\bigr)du\\
& =-\frac{\#{\cal
L}({\mathbf t})\P({\cal G}(\hat u)={\mathbf
    t})}{p_0^{(\hat u)}}F'(u)\bigl(1-\mu(\hat u)\bigr)du
\end{align*}
by Equation (\ref{eq:pzero}), which yields (\ref{PropA3}).

\medskip
Summing (\ref{PropA3}) over all finite trees $\mathbf{t}$ gives
$$\P(\#\mathcal{G}(A-)<+\infty\bigm|A=u)=\E\left[M(\hat
  u,\mathcal{G}(\hat u))\right]=1$$
by the martingale property, which is (\ref{PropA5}).

\medskip

Finally,  (\ref{PropA1}) gives
$$\P(A=\bar u)=\P(\forall u<\bar u, A>u)=\lim_{u\to\bar u}1-\bar
F(u)=F(\bar u-).$$
As $F$ is non-increasing, $F(\bar u-)$ indeed exists. Moreover, 
we have by taking the limit in (\ref{Forgu})
$$F(\bar u-)=g_{\bar u}(F(\bar u-))$$
and, by assumption (\ref{assump}), the only fixed points of $g_{\bar
  u}$  are 0 and 1, which gives (\ref{PropA1bis}).

This together with Formula (\ref{PropA1}) also give
$$
A\overset{d}{=}\bar{F}^{-1}(\gamma)\overset{d}{=}\bar{F}^{-1}(1-\gamma).
$$
Thus $(A, \bar{F}(A))\overset{d}{=}(\bar{F}^{-1}(1-\gamma),
1-\gamma).$ So we have
$$
(A, F(A))=(A, 1-\bar{F}(A))\overset{d}{=}(\bar{F}^{-1}(1-\gamma),
\gamma).
$$
which is just (\ref{PropA4}).  \qed

\end{proof}

With Remark \ref{rem:Gstar} and Proposition \ref{PropA} in
hand, we have the following representation of the ascension process
$\{{\cal G}(\alpha): 0\leq \alpha<A\}$ under assumption (\ref{assump}).

\begin{proposition}\label{PropRepA}Under assumption (\ref{assump}), we have
\bgeqn \label{PropRepA1} \{{\cal G}(u),\ 0\leq
u<A\}\overset{d}{=} \{{\cal G}^*(u \gamma): 0\leq
u<\bar{F}^{-1}(1-\gamma)\},
 \edeqn
where $\gamma$ is a r.v with uniform distribution on $(0,1)$,
independent of $\{{\cal G}^{*}(u):0\leq u\leq1\}.$ 
\end{proposition}

\begin{proof} Let  $\{{\cal G}^*(u): 0\leq
u\leq 1\}$ be independent of $A$. Then by Remark
\ref{rem:Gstar},
 $$
 \P({\cal G}^{*}(\hat{A})={\mathbf t}\bigm| A=a)=\P({\cal G}^*(\hat{a})={\mathbf
 t})=M(\hat a,\mathbf{t})\P({\cal
G}(\hat{a})={\mathbf t}).
 $$
Thus it follows from (\ref{PropA3}) that $(A, {\cal G}(A-))\overset{(d)}{=}(A,
{\cal G}^*(\hat{A})).$ On the other hand, by the definition of
node-pruning, for every $\mathbf{t}\in \mathbf{T}$, $0\le \alpha< 1$
and $0\le \beta<\bar u$,
 $$\P\bigl(({\cal G}(s\beta), 0\leq s\leq 1)\in\cdot\bigm|{\cal G}(\beta)={\mathbf
t}\bigr) = \P\bigl(({\cal G}^{*}(s\alpha), 0\leq s\leq 1)\in\cdot\bigm|{\cal
G}^{*}(\alpha)={\mathbf t}\bigr).
$$
Thus conditioning on the terminal value implies
$$
\{{\cal G}(u),\ 0\leq u<A\}\overset{d}{=} \{{\cal
G}^*(\hat{A}u/A): 0\leq u<A\}.
$$
Then (\ref{PropRepA1}) follows from (\ref{PropA4}). \qed
\end{proof}

\begin{example}(Binary case)\\
If $p=\frac{1}{2}\dz_2+\frac{1}{2}\dz_0$, then  $\bar{u}=2$
and  ${\cal G} ({u})$ is a Galton-Watson tree with binary offspring
distribution $\frac{u}{2}\dz_2+(1-\frac{u}{2})\dz_0$ for $0\leq
u\leq 2$.  In this case, we have
$$F(u)=\frac{2}{u}-1$$
and the ascension time $A$ is distributed as
$$
\bar{F}^{-1}(1-\gamma)=\frac{2}{1+\gamma}
$$
where $\gamma$ is a uniform random variable on $(0,1)$. $A$'s density is
given by
$$f(t)=-F'(t)=\frac{2}{t^2}1_{[1,2]}(t).$$
\end{example}

\begin{example}(Geometric case)\\
We suppose that the critical offspring distribution $p$ is of the form
$$p_k=(1-\beta)^2 \beta^{k-1}\ \mbox{for }k\ge 1,\qquad
p_0=\beta.$$

In that case, we have 
$$\begin{cases}
p_k^{(u)}=(1-\beta)^2(u\beta)^{k-1} & \mbox{for }k\ge 1,\\
p_0^{(u)}=1-\frac{(1-\beta)^2}{1-u\beta},
\end{cases}$$
$\bar u=2-\beta$, and assumption (\ref{assump}) is statisfied.

We then get
$$F(u)=\frac{2-u-\beta}{1-u\beta}\frac{1}{u}$$
and the ascension time $A$ has density
$$\left(\frac{2-\beta}{u^2}+\frac{(1-\beta)^2\beta}{(1-u\beta)^2}\right)\mathbf{1}_{[1,2-\beta]}(u).$$
\end{example}

\section{Proof of Lemma \ref{lem:leaves}}\label{sec:appendix}

Let $\mathcal{G}_p$ be a Galton-Watson tree with offspring
distribution $p$ such that $p_1<1$.

If $\mathbf t$ is a tree, we denote by $(a_1,a_2,\ldots,a_m)$ the numbers of
offsprings of the inner nodes. Its number of leaves is then
$$a_1+\cdots+a_m-m+1.$$

If $\mathbf t$ is a tree with $n$ leaves, we have
$$\P(\mathcal{G}_p={\mathbf t})=p_{a_1}\cdots p_{a_m}p_0^n$$
and therefore
$$\P(\#{\cal L}_p=n)=C_p(n) p_0^n$$
with
$$C_p(n)=\sum_{\mathbf{t},\ \#\mathcal{L}(\mathbf{t})=n}p_{a_1}\cdots p_{a_m}.$$

Then we have, for every $n$ such that $C_p(n)\ne 0$,
\begin{equation}\label{eq:condleaf}
\P(\mathcal{G}_p={\mathbf t}|\#{\cal L}_p=n)=\P(\mathcal{G}_q={\mathbf t}|\#{\cal L}_q=n)\iff
\frac{p_{a_1}\ldots p_{a_m}}{C_p(n)}=\frac{q_{a_1}\ldots
  q_{a_m}}{C_q(n)}.
\end{equation}

First, let us suppose that
$$\P(\mathcal{G}_p={\mathbf t}|\#{\cal L}_p=n)=\P(\mathcal{G}_q={\mathbf t}|\#{\cal L}_q=n).$$

For $n=1$, all the trees with one leaf are those with one offspring
at each generation until the last individual dies. Therefore,
\begin{align*}
\P(\mathcal{G}_p={\mathbf t}|\#{\cal L}_p=1)=\P(\mathcal{G}_q={\mathbf t}|\#{\cal L}_q=1) & \iff \forall k\ge 0,\
p_1^k(1-p_1)=q_1^k(1-q_1)\\
&\iff p_1=q_1.
\end{align*}

We set $n_0=\inf\{n\ge 2, p_n>0\}.$
We then set $u=\left(q_{n_0}/p_{n_0}\right)^{1/(n_0-1)}$.

If the only nonzero terms of $p$ are $p_0$, $p_1$ and $p_{n_0}$, the
relation
$$q_n=u^{n-1}p_n$$
is trivally true for every $n\ge 1$.

In the other cases, let $n>n_0$ such that $p_n>0$ and
let $N$ be the integer defined by:
$$N=2(n-1)(n_0-1).$$
Let us consider first a tree $\mathbf t$ that has $N+1$ leaves, $n-1$
inner nodes with $n_0$ offsprings and $n_0-1$ inner nodes with $n$
offsprings. Applying (\ref{eq:condleaf}) to that tree gives
$$\frac{p_{n_0}^{n-1}p_n^{n_0-1}}{C_p(N+1)}=\frac{q_{n_0}^{n-1}q_n^{n_0-1}}{C_q(N+1)}\cdot$$
Then, let us consider another tree with $N+1$ leaves composed of
$2(n-1)$ inner nodes with $n_0$ offsprings. For that new tree,
(\ref{eq:condleaf}) gives
$$\frac{p_{n_0}^{2(n-1)}}{C_p(N+1)}=\frac{q_{n_0}^{2(n-1)}}{C_q(N+1)}\cdot$$
Dividing the two latter equations gives
$$q_n=u^{n-1}p_n.$$

It remains to remark that this identity also holds when $n=n_0$ and
when $p_n=0$.

\bigskip
Conversely, let us suppose that $q_n=u^{n-1}p_n$ for every $n\ge
1$. Let $n$ such that $C_p(n)\ne 0$. Then, for every
${\mathbf t}$ with $n$ leaves, we have
\begin{align*}
q_{a_1}\ldots q_{a_m}
& =u^{a_1-1}p_{a_1}\ldots u^{a_m-1}p_{a_m}\\
& =u^{a_1+\cdots a_m-m}p_{a_1}\ldots p_{a_m}\\
& =u^{n-1}p_{a_1}\ldots p_{a_m}.
\end{align*}

We then have $C_q(n)=u^{n-1}C_p(n)$ and

$$\frac{q_{a_1}\ldots q_{a_m}}{C_q(n)}=\frac{u^{n-1}p_{a_1}\ldots
  p_{a_m}}{u^{n-1}C_p(n)}=\frac{p_{a_1}\ldots p_{a_m}}{C_p(n)},$$
that is
$$\P(\mathcal{G}_p={\mathbf t}|\#{\cal
  L}_p=n)=\P(\mathcal{G}_q={\mathbf t}|\#{\cal L}_q=n).$$

\bigskip
{\bf Acknowlegments}: The authors wish to thank an anonymous referee
for his remarks that improved significantly the presentation of the
paper.

Hui He also wants to thank the laboratory MAPMO for his pleasant stay
at Orl\'eans where this work was done.

\bigskip

\bigskip

\newcommand{\sortnoop}[1]{}

\end{document}